\documentclass[a4paper,12pt]{amsart}
\usepackage[a4paper, margin=1in]{geometry}
\usepackage{amsthm}
\usepackage{aliascnt}
\usepackage{amsmath,amssymb,amsfonts}
\usepackage{hyperref}
\usepackage[shortlabels]{enumitem}

\usepackage{tikz-cd}

\newcommand{\basetheorem}[3]{
	\newtheorem{#1}{#2}[#3]
	\newtheorem*{#1*}{#2}
	\expandafter\def\csname #1autorefname\endcsname{#2}
}
\newcommand{\maketheorem}[3]{
	\newaliascnt{#1}{#3}
	\newtheorem{#1}[#1]{#2}
	\aliascntresetthe{#1}
	\expandafter\def\csname #1autorefname\endcsname{#2}
	\newtheorem{#1*}{#2}
}

\basetheorem{theorem}{Theorem}{section}
\maketheorem{definition}{Definition}{theorem}
\maketheorem{lemma}{Lemma}{theorem}
\maketheorem{example}{Example}{theorem}
\maketheorem{corollary}{Corollary}{theorem}
\maketheorem{conjecture}{Conjecture}{theorem}
\maketheorem{proposition}{Proposition}{theorem}
\maketheorem{question}{Question}{theorem}
\maketheorem{remark}{Remark}{theorem}

\makeatletter
\newtheorem*{rep@theorem}{\rep@title}
\newcommand{\newreptheorem}[2]{%
	\newenvironment{rep#1}[1]{%
		\def\rep@title{#2 \ref{##1}}%
		\begin{rep@theorem}}%
		{\end{rep@theorem}}}
\makeatother

\makeatletter
\@namedef{subjclassname@2020}{%
	\textup{2020} Mathematics Subject Classification}
\makeatother

\newreptheorem{theorem}{Theorem}

\newcommand{\BS}{\mathbf{B}_{+}}

\newcommand{\Ht}[1]{H^{i}_{t}}

\newcommand{\ox}[1][X]{\mathcal{O}_{#1}}

\usepackage{xcolor}

\title{Mori Fibrations in Mixed Characteristic}

\author{Liam Stigant}

\address{Department of Mathematics, Imperial College London, 180 Queen's Gate, 
	London SW7 2AZ, UK} 
\email{l.stigant18@imperial.ac.uk}

\begin{document}

	\begin{abstract}
		This paper resolves several outstanding questions regarding the Minimal Model Program for klt threefolds in mixed characteristic. Namely termination for pairs which are not pseudo-effective, finiteness of minimal models and the Sarkisov Program. 
	\end{abstract}
		\maketitle
	\tableofcontents

	\section{Introduction}
	
	Recent work in \cite{bhatt2020globally+} establishes the bulk of the Minimal Model Program for KLT threefold pairs over suitable bases in base in mixed and positive characteristic. In particular it is shown that one can always run an MMP with scaling. When the pair is pseudo-effective more is known, it is shown that in fact every MMP terminates without the need for scaling. A small adaptation of the arguments of \cite{kawamata2008flops} ensures these models are connected by flops. This paper focuses on the outstanding questions in this setting for pairs which are not pseudo-effective. The main restrictions are that the residue fields of $R$ should have characteristic $p=0$ or $p>5$. A full characterisation of suitable base rings is given in \autoref{setup}.
		
	First it is shown that in fact the threefold MMP over a positive dimensional base always terminates, extending the termination result of \cite{bhatt2020globally+} to pairs which are not pseudo-effective.
	
	\begin{theorem}[\autoref{termination}]
		Let $X$ be an integral, normal threefold over $R$ equipped with a projective morphism $X \to T$, where $T$ is quasi-projective over $R$.
		If $(X,\Delta)$ is a threefold dlt pair over $R$ and the image of $X$ in $T$ is positive dimensional then any $K_{X}+\Delta$ MMP terminates.
	\end{theorem}

	An MMP for a pair which is not pseudo-effective will always terminate with a Mori fibre space. Unlike minimal models these are not connected by flops or even isomorphic in codimension $1$. They can be very varied even in dimension $2$. Nonetheless they are conjecturally related by a sequence of elementary transformations called Sarkisov Links. This claim is known as the Sarkisov program. It is shown that any two threefold Mori fibres spaces which are the output of the same MMP are related by Sarkisov links.
	
	\begin{theorem}[\autoref{sarkisov}]
		Fix an integral quasi-projective scheme $T$ over $R$. Let $g_{1}:Y_{1} \to Z_{1}$ and $g_{2}:Y_{2} \to Z_{2}$ be two Sarkisov related, klt Mori fibre spaces of dimension $3$, projective $T$. If the $Y_{i}$ have positive dimension image in $T$, then they are connected by Sarkisov links.
	\end{theorem}
	
	The proof of this second theorem follows closely the work of \cite{hacon2009sarkisov}. The main technical work comes in proving a suitable version of finiteness of minimal models.
	
	\begin{theorem}[\autoref{rltfiniteness}]
		Let $X$ be an integral, normal threefold over $R$ equipped with a projective morphism $X \to T$, where $T$ is quasi-projective over $R$ and the image of $X$ in $T$ is positive dimensional. Let $A$ be an ample $\mathbb{Q}$-Cartier divisor and $C$ be a rational polytope inside $\mathcal{L}_{A}(V)$. Suppose there is a boundary $A+B \in \mathcal{L}_{A}(V)$ such that $(X,A+B)/T$ is a klt $R$-pair. Then the following hold:
	\begin{enumerate}
	\item There are finitely many birational contractions $\phi_{i}:X \dashrightarrow Y_{i}$ such that 
	\[\mathcal{E}(C) = \bigcup \mathcal{W}_{i}=\mathcal{W}_{\phi_{i}}(C)\]
	where each $\mathcal{W}_{i}$ is a rational polytope. Moreover if $\phi:X \to Y$ is a wlc model for any choice of $\Delta \in \mathcal{E}(C)$ then $\phi=\phi_{i}$ for some $i$, up to composition with an isomorphism.
	
	\item There are finitely many rational maps $\psi_{j}:X \dashrightarrow Z_{j}$ which partition $\mathcal{E}(C)$ into subsets $\mathcal{A}_{\psi_{j}}(C)=\mathcal{A}_{i}$.
	\item  For each $W_{i}$ there is a $j$ such that we can find a morphism $f_{i,j}: Y_{i} \to Z_{j}$ and $W_{i} \subseteq \overline{A_{j}}$.
	\item  $\mathcal{E}(C)$ is a rational polytope and $A_{j}$ is a union of the interiors of finitely many rational polytopes.
	\end{enumerate}
	If $C$ contains only klt boundaries then $A$ big suffices.
	\end{theorem}
	
	In fact these results hold for a slightly more general class of singularities - rlt pairs, which are essentially pairs which are replaceable by linearly equivalent klt pairs locally over the base. This generalisation is necessary due to the lack of appropriate Bertini type theorems over a general ring. Even if one starts with Mori Fibre Spaces coming from a klt MMP, the Sarkisov links may involve rlt pairs. A full definition of rlt is given in \autoref{rlt-section} and a description of Sarkisov links in \autoref{Sarkisov-section}.
	
	The requirement that the base be positive dimensional is mostly out of an abundance of caution. It is likely the results of \cite{DW19} are sufficient to carry out the arguments needed for \autoref{rltfiniteness} over an $F$-finite field. Over a perfect field of positive characteristic this theorem is already known due to \cite{das2020finiteness} and in characteristic $0$ due to \cite{shokurov2011geography}. The result is also known in higher dimensions over characteristic $0$ fields by \cite{birkar2010existence}.
	
	\textbf{Acknowledgments}
	Thanks to Federico Bongiorno and Paolo Cascini for their support and for many useful discussions. Thanks also to the EPSRC for my funding.	
	
	\section{Preliminaries}
	
	We will say that $f\colon X \to Y$ is a contraction if $f_{*}\ox=\ox[Y]$.
	
	\begin{definition}\label{log-pair}
		
		A sub-log pair $(X,\Delta)$  with $\mathbb{K}$ boundary is an excellent, Noetherian, integral, normal scheme $X$ admitting a dualising complex together with an $\mathbb{K}$-divisor $\Delta$ such that $(K_{X}+\Delta)$ is $\mathbb{K}$-Cartier. If $\Delta$ is effective, we say $(X,\Delta)$ is a log pair. When $\Delta=0$ we just say $X$ is a log pair.
		
	\end{definition}
	
	We adopt the notation and definitions of \cite{KM}[Section 2.3, Definition 2.34] for singularities of pairs. In particular for $E$ with centre on $X$ we denote the discrepancy by $a(E,X,\Delta)$. If $\Delta=0$ we write $a(E,X)$ for brevity.
	
	\begin{definition}\label{setup}
		An $R$-pair $(X,\Delta)/ T$ with $\mathbb{K}$-boundary will be the following data:
		\begin{itemize}
			\item An excellent, normal ring $R$ of mixed characteristic and finite dimension which admits a dualising complex and whose residue fields have characteristic $p=0$ or $p>5$;
			\item A integral, normal quasi-projective $R$-variety $T$;
			\item A normal, integral scheme $X$;
			\item A projective contraction $f \colon X \to T$; and
			\item A $\mathbb{K}$-divisor $\Delta \geq 0$ with $(K_{X}+\Delta)$ $\mathbb{K}$-Cartier, for $\mathbb{K}=\mathbb{R}$ or $\mathbb{Q}$.
		\end{itemize}
		
		The dimension of such a pair is the dimension of $X$. Equally the pair is said to $\mathbb{Q}$-factorial if $X$ is. If no $\mathbb{K}$ is specified we default to $\mathbb{K}=\mathbb{R}$. As with log pairs, we drop $\Delta$ from the notation when $\Delta=0$.
		
		In this paper we consider exclusively $T$ with $\dim T \geq 1$.
	\end{definition}

	The assumptions that $X \to T$ is a contraction and $T$ is integral, normal can be dropped in applications since we can always Stein factorise an arbitrary projective $X \to T$. In this case it is necessary to assume the image of $X$ has dimension at least $1$ in place of $T$.
%
%
%
%
%

	\begin{theorem}[Cone Theorem]\label{cone-thm}
		Let $(X,\Delta)/T$ be an lc $\mathbb{Q}$-factorial threefold $R$-pair. Then there is a countable collection of curves $\{C_{i}\}$ on $X$ such that:
		\begin{enumerate}
			\item $$\overline{NE}(X/T)=\overline{NE}(X/T)_{K_{Y}+\Delta \geq 0} + \sum_{i} \mathbb{R}[C_{i}]$$
			\item The rays $C_{i}$ do not accumulate in $(K_{Y}+\Delta)_{<0}$.
			\item For each $i$ there is $d_{C_{i}}$ with 
			\[0 < -(K_{X}+\Delta).C_{i} \leq 4d_{c_{i}}\]
			and $d_{C_{i}}$ divides $L\cdot_{k}C_{i}$ for every Cartier divisor $L$ on $X$.
		\end{enumerate}
	\end{theorem}

	\begin{proof}
		
		This follows immediately from \cite[Theorem H]{bhatt2020globally+} after taking a dlt modification, which exists by \cite[Corollary 9.19]{bhatt2020globally+}.
		
	\end{proof}
	
	\begin{theorem}[Basepoint Free Theorem]\cite[Theorem G]{bhatt2020globally+}\label{bpt}
		Let $(X,\Delta)/T$ be a klt $R$-pair of dimension $3$. Let $L$ be an $f$-nef divisor with $L-(K_{X}+\Delta)$ is $f$-big and nef. Then $L$ is semiample. 
	\end{theorem}
	
	\begin{theorem}\cite[Theorem F]{bhatt2020globally+}
		
		Let $(X,\Delta)/T$ be a dlt $R$-pair of dimension $3$. We can run a $K_{X}+\Delta$ MMP. If $K_{X}+\Delta$ is pseudo-effective then this MMP terminates. Further an MMP with scaling by an ample divisor will terminate for any choice of pair.
		
	\end{theorem}
	
	Termination for pseudo-effective pairs in this setting is assured by the following theorem, together with non-vanishing on the generic fibre.
	
	\begin{theorem}\cite[Proposition 9.20]{bhatt2020globally+}\label{partial-termination}
		Let $(X,\Delta)\to T$ be a threefold dlt $R$-pair. Suppose that
		$$(X,\Delta)=:(X_{0},\Delta_{0}) \dashrightarrow (X_{1},\Delta_{1}) \dashrightarrow$$
		is a sequence of $(K_{X}+\Delta)$ flips. Then neither the flipped nor the flipping locus are contained the support of $\Delta_{n}$ for all sufficiently large $n$.
	\end{theorem}

	We will also need the following construction, essentially due to \cite{mumford1961topology}.
	
	\begin{lemma}
		
		Let $\pi:X \to Y$ be a projective contraction from a regular scheme to a normal scheme, both of dimension $2$. Let $E_{1},...,E_{n}$ be the exceptional curves. Choose a divisor $D$ on $Y$ and write $D'$ for the strict transform of $D$. Then there are unique $m_{i} \geq 0$ with $D'+\sum m_{i}E_{i} \equiv_{Y} 0$. If $D$ is $\mathbb{Q}$-Cartier then we have $\pi^{*}D= D'+\sum m_{i}E_{i}$.
		
	\end{lemma}
	
	\begin{proof}
		
		By \cite[Theorem 10.1]{kk-singbook}, the intersection form $[E_{i}.E_{j}]$ is negative definite. Hence there is a unique choice of $m_{i}$ with $D'+\sum m_{i}E_{i} \equiv_{Y} 0$. It remains to show that $m_{i} \geq 0$. By \cite[Lemma 10.2]{kk-singbook} there is $E= \sum r_{i}E_{i}$ effective on $X$ with $-E$ ample over $Y$. Then $E.E_{i} < 0$ for each $i$ ensures that $r_{i} > 0$ for all $i$.

		Now suppose for contradiction that $m_{k} < 0$ for some $k$. Then we may suppose that $m_{k}/r_{k}$ is minimal, otherwise if $m_{j}/r_{j}$ is minimal we just replace $k$ with $j$ as we must still have $m_{j} < 0$. We must have, for every $j$, that $D'.E_{j}\geq 0 $ as it does not contain any $E_{j}$ and thus as $D' \equiv_{Y} - \sum m_{i}E_{i}$ we have
		
		\[0\geq (\sum_{i} m_{i} E_{i}).E_{j} = \sum_{I} \frac{m_{i}}{r_{i}}(r_{i}E_{i}.E_{j}) \geq \frac{m_{k}}{r_{k}} \sum_{i} (r_{i}E_{i}.E_{j}) > 0\]
		
		This is a contradiction and hence in fact $m_{i} \geq 0$ for each $i$. That this agrees with the pullback when $D$ is $\mathbb{Q}$-Cartier is immediate from uniqueness.
		
	\end{proof}
	
	\begin{lemma}\label{num-pull}
		
		Let $X$ be an $\mathbb{Q}$-factorial scheme together with a projective morphism $f:X \to Y$  with geometrically connected fibres to an excellent normal scheme of dimension $2$. Suppose $V$ is a closed subscheme of $X$ with $f(V)$ contained in a divisor $D$. Then there is a divisor $D'$ on $X$ lying over $D$, numerically trivial over $Y$ and containing $V$.
		
		\end{lemma}
	
	\begin{proof}
		
		Let $\pi \colon Y' \to Y$ be a resolution of $Y$ and $X'$ be the normalisation of the dominant component of the fibre product $X\times_{Y} Y'$. From above we have $F$ on $Y$ lying over $D$ with $F \equiv_{Y} 0$. We have induced maps $g\colon X' \to Y'$ and $\phi\colon  X' \to X$. Now $g^{*}F$ is numerically trivial over $Y$, and hence over $X$. Thus as $X$ is $\mathbb{Q}$-factorial there is $D'$ with $\phi^{*}D'=F$. It is clear from the construction that $f_{*}D'=\pi_{*}F=D$. Suppose that $C$ is a curve lying over $D$, then we must have $D'.C =0$. If $C$ is not contained in $D'$ then since $f$ has connected fibres we may suppose that $D'$ meets $C$, up to replacing $C$ with another curve in the same fibre, but then $D'.C > 0$, a contradiction. Hence $D'$ contains every curve, and hence every fibre, over $D$. In particular it contains $V$.

	\end{proof}

	\section{Termination}
	
	In this section we study termination for threefold pairs over positive dimensional bases. In this setting we will show that every $K_{X}+\Delta$ MMP terminates for a dlt pair $(X,\Delta)/T$. We rely heavily on \autoref{partial-termination}. The key remaining argument is if $X \to T$ is a klt pair then there is an open set on which every contraction is horizontal.
	
	We prove this by reducing to the case that $(X,\Delta)$ is terminal. In mixed and positive characteristic this then follows from the liftablity of $-1$ curves, see \cite{katsura1985elliptic}. This argument does not work in purely positive characteristic but provides motivation for our approach. Instead we adapt a termination argument for terminal pairs, largely due to Shokurov \cite{shokurov1986nonvanishing}. If $X \to T$ is projective and $U\subseteq T$ is an open set we will write $X_{U}=X\times_{T}U$ and $\Delta_{U}=\Delta|_{X_{U}}$.

	\begin{definition}
		
		Let $X$ be a terminal threefold log pair quasi-projective over $R$. We define the difficulty 
		\[d(X)=\#\{E\colon a(E,X) < 1\}\]
		this is finite by \cite[Proposition 2.36]{KM}, since log resolutions exist by \cite[Proposition 2.12]{bhatt2020globally+}.
		
	\end{definition}
	
	Clearly if $Y\hookrightarrow X$ is an open immersion then $d(Y) \leq d(X)$ since every valuation with centre on $Y$ is also a valuation with centre on $X$.
	If $X \dashrightarrow X'$ is a $K_{X}$ flip then $d(X') \leq d(X)$ by \cite[Lemma 3.38]{KM}. We claim in fact this inequality is strict.
	
	\begin{lemma}
		
		Let $X/T$ be a terminal threefold $R$-pair and $X \dashrightarrow X'$ a $K_{X}$ flip, then $d(X') < d(X)$.  
		
	\end{lemma}
	
	\begin{proof}
		It suffices to find a divisor $E$ with $a(E,X) <1$ and $a(E,X') \geq 1$.
		Let $C'$ be an irreducible component of the flipped curve. Then $X'$ is terminal, so it is smooth at the generic point $P$ of $C$ by \cite[Corollary 2.30]{kk-singbook}. Let $Y \to X$ be the blowup of $C'$ and $E$ the dominant component of the exceptional divisor. By localising at $P$ we see that $a(E,X')=1$, since this is the blowup of a smooth point on a surface.
		
		Let $C$ be the centre of $E$ on $X$. Then $C$ is a component of the flipping curve and so we have $a(E,X) < a(E,X')$ by \cite[Lemma 3.38]{KM} concluding the proof.	
	\end{proof}
	
	\begin{theorem}\label{l-open-term}
		
		Let $(X,\Delta)/T$ be a terminal threefold $R$-pair. Then there is an open set $U \subseteq T$ such that every $K_{X_{U}}+\Delta_{U}$ negative contraction is a horizontal divisorial contraction.
	\end{theorem}
	
	\begin{proof}
		
		Write $\Delta=\sum_{1}^{n} a_{k}D_{k}$, we argue by induction on $n$. Suppose first that $n=0$ and for contradiction there is no such $U$. Thus we have a sequence of non-empty open sets $U_{i} \subseteq U_{i-1}$ such that there is a $K_{X_{U_{i}}}$ negative extremal ray $L_{i}$ supported away from $U_{i+1}$. We write $X_{i}=X\times U_{i}$.
		
		If $L_{i}$ induces a divisorial contraction $f_{i} \colon X_{i} \to X'_{i}$ then $\rho(X_{i+1}) \leq \rho(X'_{i}) < \rho(X_{i})$ since $f_{i}$ is an isomorphism over $U_{i+1}$. Similarly if $L_{i}$ induces a flip $f_{i} \colon X_{i} \dashrightarrow X'_{i}$ then $d(X_{i+1}) \leq d(X'_{i}) < d(X_{i})$. Since both are positive integers there can be only finitely many such $U_{i}$, a contradiction.
		
		Now suppose $n>0$.	Let $\Delta^{n-1}=\sum_{1}^{n-1}a_{i}D_{i}$ then by induction there is an open set $U\subseteq T$ such that every $K_{X_{U}}+\Delta^{n-1}_{U}$ negative contraction is a horizontal divisorial contraction. If $D_{n}$ is not horizontal, we can shrink $U$ so it doesn't meet the image of $D_{n}$ and the result follows immediately. This gives the result if $\dim T=3$.
		
		Otherwise let $S$ be the normalisation of $D_{n}$. If $\dim T=2$ then there is an open set $V$ of $T$ on which $S_{V} \to T$ is finite, and hence of relative Picard rank $0$. In particular $S_{V}$ contains no curves. If $\dim T=1$ then by \cite[Lemma 2.13]{tanaka2018minimal} there is an open set $V$ of $T$ such that $S_{V}$ has relative Picard rank $1$. 
		
		In either case, replace $U$ with $U \cap V$, then $X,S$ with $X_{U},S_{U}$ and $\Delta$ with $\Delta|_{X_{U}}$. It suffices to show that every extremal $K_{X}+\Delta$ negative contraction is a horizontal divisorial contraction. Suppose for contradiction $L$ is an extremal ray inducing one that is not. We must have $D_{n}.L <0$ from our choice of $U$. Thus induced contraction restricts to a nontrivial birational morphism $S\to S'$ say. However $S$ has Picard rank at most $1$, so the only possibility is this map contracts $S$ entirely. In particular this defines a horizontal divisorial contraction, a contradiction. The claim follows.
		
	\end{proof}
	
	We can extend this immediately to klt pairs.
	
	\begin{theorem}\label{c-open-klt}
		Let $(X,\Delta)/T$ be a terminal threefold $R$-pair. Then there is an open set $U \subseteq T$ such that every $K_{X_{U}}+\Delta_{U}$ negative contraction is a horizontal divisorial contraction.
	\end{theorem}
	
	\begin{proof}
		
		Let $\pi \colon (Y,\Delta_{Y}) \to (X,\Delta)$ be a terminalisation, which exists by \cite[Proposition 9.17]{bhatt2020globally+}. Then by \autoref{l-open-term} there is an open set $U\subseteq T$ over which every $K_{Y_{U}}+\Delta_{Y_{U}}$ negative contraction is divisorial. We claim the same holds for $K_{X_{U}}+\Delta_{U}$ negative contractions. 
		
		Indeed if $f \colon X_{U} \to Z$ is any such contraction then $K_{Y_{U}}+\Delta_{Y_{U}}$ is not nef over $Z$. In particular we get a contraction $g \colon Y_{U} \to Z$, which is necessarily a horizontal divisorial contraction. In particular $g$ is not an isomorphism over the generic point $\nu$ of $T$. However then neither can $f$ be, else $K_{Y_{\nu}}+\Delta_{Y_{\nu}}$ would be nef over $Z_{\nu}$. Thus $f$ is a horizontal divisorial contraction as claimed.
		
	\end{proof}
	
	\begin{corollary}\label{termination}
		Let $f:(X,\Delta) \to T$ be a $\mathbb{Q}$-factorial threefold dlt pair over $R$, then any $K_{X}+\Delta$ MMP terminates.
	\end{corollary}
	
	\begin{proof}
		
		It is enough to show there is no infinite sequence of flips. Note that \autoref{partial-termination} ensures that the flipping and flipped curves are eventually disjoint from $\lfloor \Delta \rfloor$. Therefore, replacing $\Delta$ with $\Delta-\lfloor \Delta \rfloor$, we may assume $(X,\Delta)$ is klt. 
		
		By \autoref{c-open-klt}, there is always some divisor $D$ on $T$ such that all the flips take place over $D$.	If $T$ is $\mathbb{Q}$-factorial then $(X,\Delta'=\Delta+tf^{*}D)$ is klt for small $t>0$ and a $K_{X}+\Delta$ MMP is also a $K_{X}+\Delta'$ MMP. Since all the flips are contained in the support of $\Delta'$ the sequence must terminate. Otherwise we must have $\dim T =2$ so we use \autoref{num-pull} in place of pulling back $D$ and conclude exactly as above.
		
	\end{proof}

	\section{Relatively Log Terminal Pairs} \label{rlt-section}

	He	Here we introduce relatively log terminal pairs, which are essentially pairs which are replaceable by a klt pair locally over the base, and verify that the main results of the MMP extends to this setting. A suitable Bertini type theorem is also established. In this section $T$ will always be positive dimensional, in any case the results would be superfluous if $T$ were the spectrum of a field.
	
	\begin{definition}
		We say an $R$-pair $(X,\Delta)/T$ is relatively log terminal (rlt) (resp. relatively log canonical (rlc)) if there is a finite open cover $U_{i}$ of $T$ such that on each $X_{i}=U_{i} \times X$ we have $(K_{X}+\Delta)|_{U_{i}} \sim_{\mathbb{R}} K_{X_{i}}+\Delta_{i}$ where $(X_{i},\Delta_{i})$ is a klt (resp. lc) pair. In this case we say that $(X,\Delta)$ is witnessed by $(X_{i},\Delta_{i})$. We also sometimes say $\Delta$ is witnessed over $U_{i}$. 
		
		If $S \subseteq WDiv(X)$ then we say $(X,\Delta)$ is rlt (resp. rlc) with witnesses in $S$ if $\Delta_{i} \in S|_{U_{i}}$ for each $i$ for some choice of witnesses.
	\end{definition}
	\begin{remark}
		$T$ is always quasi-compact so this is equivalent to asking for $K_{X}+\Delta \sim K_{X_{p}}+\Delta_{p}$ with $(X_{p},\Delta_{p})$ klt for each $p \in T$ where $X_{p}=X \times T_{p}$ for $T_{p}$ the localisation at $p$.
	\end{remark}
	
	Being rlt can be quite a sensitive condition. In particular it's not true that if $B \leq B'$ and $(X,B')$ is rlt that $(X,B)$ must be rlt. For example, for any choice of $B$ and sufficiently ample $H$, on $X$ klt and $\mathbb{Q}$-factorial, we have that $(X,B+H)$ is rlt, though $B$ might not be. It fits well in the context of polytopes however as if $B_{i}$ are rlt then so is $\sum_{1}^{n} \lambda_{i}B_{i}$ for any choices of $\lambda_{i} \geq 0$ with $\sum \lambda_{i} \leq 1$.
	
	The pseudo-effective cone is the closure of the big cone, and $D$ is big if and only if its pullback to the generic fibre of $X \to T$ is. Hence if $U_{i}$ is any open cover of $T$, then $D$ is pseudo-effective if and only if $D|_{U_{i}}$ is. In particular an rlc pair is pseudo-effective (resp. big) if and only if its witnesses are.
	
	\begin{definition}
		Let $\phi:X \dashrightarrow Y$ be a birational contraction. Take a divisor $D$ and write $D'=\phi_{*}D$. 
		
		We say it is $D$-non-positive (resp. $D$-negative) if there is a common resolution $p:W \to X$, $q:W \to Y$ where 
		
		\[p^{*}D=q^{*}D'+E\]
		and $E \geq 0$ is $q$ exceptional (resp. $E \geq 0$ is $q$ exceptional and contains the strict transform of every $\phi$ exceptional divisor in its support). 
		
		If $(X,\Delta)$ is a pseudo-effective lc pair then $\phi$ is a weak log canonical (wlc) model if $\phi$ is $K_{X}+\Delta$ non-positive with $K_{Y}+\Delta_{Y}$ nef, where $\Delta_{Y}=\phi_{*}\Delta$. As $\phi$ is non-positive $(Y,\Delta_{Y})$ is always lc and if $(X,\Delta)$ is klt then so is $(Y,\Delta_{Y})$. 
		
		If in fact $\phi$ is $K_{X}+\Delta$ negative, $Y$ is $\mathbb{Q}$-factorial, and $(Y,\Delta_{Y})$ is dlt then $\phi$ is a log terminal model. Again if $(X,\Delta)$ is dlt then the dlt condition on $(Y,\Delta_{Y})$ is automatic as $\phi$ is negative. If $K_{Y}+\Delta_{Y}$ is semiample then $\phi$ is said to be a good log terminal model.
		
		If instead $\phi:X \dashrightarrow Y$ is a rational map then it is an ample model for $D$ if there is $H$ ample on $Y$ such that $p^{*}D\sim_{\mathbb{R}}q^{*}H+E$ where $E \geq 0$ is such that $E \leq B$ for any $p^{*}D \sim_{\mathbb{R}} B \geq 0$.
	\end{definition}

	Note that ample models are unique. Indeed if $X \dashrightarrow Y$ and $X \dashrightarrow Z$ are two ample models, then on some common resolution $W$ of both maps we have $f:W \to Y$, $g:W \to Z$ and $h:W \to X$. Now there are ample divisors $A_{Y}$, $A_{Z}$ with $f^{*}A_{Y}+E_{Y}\sim_{\mathbb{R}}h^{*}D \sim_{\mathbb{R}}g^{*}A_{Z}+E_{Z}$. But by definition $E_{Z}=E_{Y}$ and hence $f^{*}A_{Y}\sim_{\mathbb{R}}g^{*}A_{Z}$, so there is an isomorphism $i:Z \to Y$ with $i \circ f= g$ as required.
	
	If $(X,\Delta)$ is a pair then we say $\phi:X \dashrightarrow Y$ is an ample model of $(X,\Delta)$ if it is an ample model for $K_{X}+\Delta$. We can often replace pairs with linearly equivalent versions.
	
	\begin{lemma}\label{equiv}\cite[Lemma 3.6.8]{birkar2010existence}
		Let $\phi:X \to Y$ be a rational map. Suppose $(X,\Delta)$ and $(X,\Delta')$ are two pairs and $D,D'$ two $\mathbb{R}$-Cartier divisors on $X$. Take $t >0$ a positive real number.
		\begin{itemize}
			\item If $D \equiv tD'$ and $\phi_{*}D$, $\phi_{*}D'$ are both $\mathbb{R}$-Cartier then $\phi$ is $D$ negative (resp $D$ non-negative) if and only if it is $D'$ negative (resp. non-negative)
			\item If both pairs are lc and $K_{X}+\Delta \sim_{\mathbb{R}} t(K_{X}+\Delta')$ then $\phi$ is a wlc model for $(X,\Delta)$ if and only if it is a wlc model for $(X,\Delta')$.
			\item If both pairs are klt and $K_{X}+\Delta \equiv t(K_{X}+\Delta')$ then $\phi$ is a log terminal model for $(X,\Delta)$ if and only if it is a log terminal model for $(X,\Delta')$.
			\item If $D\sim_{\mathbb{R}} tD$ then $\phi$ is an ample model for $D$ if and only if it is an ample model for $D'$.
		\end{itemize}
	\end{lemma}
	
	In particular these definitions extend naturally to rlt pairs as follows.
	
	\begin{definition}
		Let $\phi:X \dashrightarrow Y$ be a rational map. If $U_{i}$ is an open cover of $T$ we write $\phi_{i}:X_{i}\dashrightarrow Y_{i}=Y\times U_{i}$.
		If $(X,\Delta)$ is a pseudo-effective rlc pair witnessed by $(X_{i},\Delta_{i})$ then $\phi$ is a weak log canonical (wlc) model of $(X,\Delta)$ if $\phi_{i}$ is an $(X,\Delta_{i})$ wlc model for each $i$. Equally if $(X,\Delta)$ is rlt then $\phi$ is a log terminal model of $(X,\Delta)$ if and only if each $\phi_{i}$ is a log terminal model of $(X_{i},\Delta_{i})$.
	\end{definition}
	
	By Lemma \ref{equiv} these definitions are independent of the choice of witnesses. In particular if $(X,\Delta)$ is lc then the definition of wlc models agrees with usual one, equally if it is klt then the definition of log terminal model is unchanged.
	
	\begin{remark}
		The usual definition of ample model works here with no modification, it is equivalent to asking for it to be an ample model for the witnesses.
	\end{remark}
	
	\begin{lemma}\label{bertini}
		Let $(X,\Delta)/T$ be an rlt $R$-pair. Take $A\geq 0$ big and nef, then $(X,\Delta+A)$ is rlt. Moreover if $D$ is a divisor on $X$ sharing no components with the augmented base locus $\BS(A)$ nor any witness of $(X,\Delta)$ then we may assume no witness of $(X,\Delta+A)$ shares a component with $D$.
	\end{lemma}
	\begin{proof}
		Write $A \sim A'+E$ for $A'$ ample and $E \geq 0$. We may assume $E$ is arbitrarily small, by writing $A\sim \delta A' + (1-\delta)A=\delta E$ and replacing $A'$ with $\delta A' + (1-\delta)A$. Thus we may suppose $(X,\Delta+E)$ is rlt such that no witnesses shares a component with $D$ and reduce to the case $A$ is ample.
		
		Pick a point $P \in T$ and localise. Write $X_{P}=X \times T_{P}$, $\Delta_{P}$ for the witness over $P$ and $D_{P}$ for the restriction of $D$. Let $\pi:Y \to X_{P}$ be a log resolution of $(X_{P},\Delta_{p}+D)$. Let $D'=\text{Supp}(\pi^{-1}_{*}D)$ and take $-$ effective, exceptional and anti-ample over $X_{p}$. So $A'=\pi^{*}A_{p}-E$ is ample. Write $K_{Y}+\Delta'=\pi^{*}(K_{X_{p}}+\Delta)$.
		
		By \cite[Theorem 2.11]{bhatt2020globally+} we can choose $A'\geq 0$ with $(Y,\Delta'+A'+E)$ klt and $(Y,\Delta'+A'+E+D')$ lc. In particular this choice of $A'$ cannot share a component with $D'$. Now $(X_{P},\Delta_{P}+\pi_{*}A')$ is klt and $\pi_{*}A'$ shares no components with $D$. Then this pair lifts to klt pair over some neighbourhood of $p$. The result follows by quasi-compactness.
	\end{proof}

	The MMP for these pairs lifts naturally from the klt case. We work in the setting of \cite{bhatt2020globally+}, however the rlt (resp. rlc) case always follows from corresponding results for klt (resp. lc) pairs.
	
	\begin{theorem}[rlc Cone Theorem]
		Let $(X,\Delta)$ be an rlc $\mathbb{Q}$-factorial threefold pair $R$-pair with $\mathbb{R}$ boundary. Then there is a countable collection of curves $\{C_{i}\}$ on $X$ such that:
		\begin{enumerate}
			\item $$\overline{NE}(X/U)=\overline{NE}(X/U)_{K_{Y}+\Delta \geq 0} + \sum_{i} \mathbb{R}[C_{i}]$$
			\item The rays $C_{i}$ do not accumulate in $(K_{Y}+\Delta)_{<0}$.
			\item There is an integer $M$ such that for each $i$ there is $d_{C_{i}}$ with 
			\[0 < -(K_{X}+\Delta).C_{i} \leq Md_{c_{i}}\]
			and $d_{C_{i}}$ divides $L\cdot_{k}C_{i}$ for every Cartier divisor $L$ on $X$.
		\end{enumerate}
	\end{theorem}

	\begin{proof}
	For ease of notation we will often view cycles on $X_{i}$ as cycles on $X$ without renaming.
	
	Suppose that $(X,\Delta)$ has witnesses $(X_{i}=X \times U_{i}, \Delta_{i})$ for some open cover $U_{i}$ of $T$. Then $U_{i}$ is still quasi-projective over $R$ and the Cone Theorem holds for each $(X_{i},\Delta_{i})$. Let $\gamma_{i,j}$ be the $K_{X_{i}}+\Delta_{i}$ negative extremal curves. These are also $K_{X}+\Delta$ negative, though they need not be extremal on $X$.
	
	Suppose now that $R$ is a $K_{X}+\Delta$ negative extremal ray. Let $r\in R$ be a non-zero cycle. Then $r$ is the limit of some effective cycles $r^{k}$. We write $r_{i}^{k}$ for the part of $r$ supported over $U_{i}$. Then $r_{i}=\lim r^{k}_{i}$ is still pseudo-effective, moreover $r-r_{i}=\lim r^{k}-r^{k}_{i}$ is also. Since $R$ is extremal we must have for each $i$ that either $r_{i}=0$ or $r=t_{i}r_{i}$ for some $t_{i} > 0$. There must be some $i$ with $r_{i} \neq 0$, else we would have $r=0$. However $r_{i}$ then generates an extremal $K_{X_{i}}+\Delta_{i}$ negative ray, hence $r=t_{i}r_{i}=t\gamma_{i,j}$ for some $j$ and some $t>0$. Thus the $\gamma_{i,j}$ generate all the $K_{X}+\Delta$ negative extremal rays. $(1)$ and $(3)$ follow immediately by \autoref{cone-thm}. Since there are finitely many $U_{i}$ if the rays accumulated on $X$ we could chose a subsequence consisting of extremal rays coming from some $X_{i}$ which would then accumulate on $X_{i}$, thus $2$ also holds.
\end{proof}

	\begin{theorem}[rlt Basepoint Free Theorem]
		Let $(X,\Delta)$ be a $\mathbb{Q}$-factorial threefold rlt $R$-pair with $\mathbb{R}$-boundary. Let $L$ be a nef Cartier divisor over $T$ such that $L-(K_{X}+\Delta)$ is big and nef over $T$. Then $L$ is semiample.
	\end{theorem}
	
	\begin{proof}
		This is immediate from the klt case, \cite{bhatt2020globally+}[Theorem 9.26], since semi-ampleness is local on the base and if $L-(K_{X}+\Delta)$ is big and nef over $T$ then $L_{X_{i}}-(K_{X_{i}}+\Delta_{i})$ is big and nef over $U_{i}$ for each $i$.
	\end{proof}

	\begin{theorem}[Existence of rlt flips]
	Let $(X,\Delta)/T$ be a threefold rlt $R$-pair with $\mathbb{R}$-boundary. Suppose $X \to Y$ is a flipping contraction over $T$ then the flip $X \dashrightarrow X^{+}$ exists. 
\end{theorem}

\begin{proof}
	Let $\phi:X \to Y$ be a flipping contraction for an rlt pair $(X,\Delta)$. Suppose $(X,\Delta)$ is witnessed by $(X_{i},\Delta_{i})$ and let $\phi_{i}:X_{i} \to Y_{i}$ be the induced morphism $U_{i}$. Then $\phi_{i}$ is either still a flipping contraction or an isomorphism. If $\phi_{i}$ is a flipping contraction, then the existence of flip $X_{i}^{+}$ is ensured by \cite{bhatt2020globally+}[Theorem 9.12], otherwise we take simply take $X_{i}^{+}=X_{i}$. Hence we have a suitable $X_{i}^{+}$ for each $i$. Since flips are unique these $X_{i}^{+}$ glue to a variety $X^{+}$ over $T$ such that $X \dashrightarrow X^{+}$ is the required flip.
\end{proof}

\begin{theorem}[Termination of rlt flips]\label{rlt-term}
	Let $(X,\Delta)/T$ be a threefold rlt $R$-pair with $\mathbb{R}$-boundary. Then any sequence of $(K_{X}+\Delta)$ flips terminates.
\end{theorem}

\begin{proof}
	
	 Let $f^{i}:X^{i} \to X^{i+1}$ be a sequence of flips from $X=X^{0}$ of an rlt pair $(K_{X}+\Delta)$. Then $(K_{X}+\Delta)$ is witnessed over some finite open cover $U_{j}$ and the restriction $f^{i}_{j}:X_{j}^{i} \to X_{j}^{i+1}$ is a sequence of flips and isomorphisms for the klt pair $(K_{X_{j}}+\Delta_{j})$ for each $j$. In particular for fixed $j$ the sequence eventually terminates by \autoref{termination}, but then as there are finitely many $j$, the global sequence $f^{i}$ also terminates.
\end{proof}

\begin{theorem}[MMP for rlt pairs]\label{rltmmp}
	Let $(X,\Delta)/T$ be a threefold rlt $R$-pair with $\mathbb{R}$-boundary, then we can run a $K_{X}+\Delta$ MMP. If $K_{X}+\Delta$ is pseudo-effective then this terminates with a good log terminal model, otherwise it ends in a Mori fibre space.
\end{theorem}

\begin{proof}
	Existence of the claimed MMPs and their termination is immediate from the above results. Suppose then $\phi:X \dashrightarrow Y$ is a log terminal model, since semiampleness is checked locally over the base we can assume that $(X,\Delta)$ is klt. Then $K_{Y}+\Delta_{Y}$ is a good log terminal model by \cite[Theorem 1.1]{bernasconi2021abundance}.
\end{proof}

\section{RLT Polytopes}

	In this section we introduce rlt versions of Shokurov Polytopes and provide some key technical results for their usage in the proof of Finiteness of Minimal Models. In particular we show that $\mathcal{RL}_{A}(V)$ is in fact a rational polytope. In this section, as in \autoref{setup}, $R$ will always be an excellent ring with dualising complex, $T$ will be a positive dimensional, quasi-projective $R$ scheme and $X$ will always be an integral scheme admitting a projective contraction $X \to T$. All pairs will be considered as $R$ pairs over $T$.
	
	\begin{definition}
	Fix a $\mathbb{Q}$-divisor $A\geq 0$. Let $V$ be a finite dimensional, rational affine subspace of $WDiv_{\mathbb{R}}(X)$ containing no components of $A$. Such $V$ is called a coefficient space (for A).
	
	We have the following.
	\[V_{A}= \{A+B: B \in V\}\]
	\[\mathcal{L}_{A}(V)=\{\Delta=A+B \in V_{A}: (X,\Delta)/T \text{ is an lc pair}\}\]
	\[\mathcal{RL}_{A}(V)=\{\Delta=A+B \in V_{A}: (X,\Delta)/T \text{ is an rlc pair with witnesses in } V_{A}\}\]
	
	We call a polytope $C$ inside $\mathcal{RL}_{A}(V)$ rlt if it is rational and contains only boundaries of rlt pairs.
	
	If $C \subseteq \mathcal{RL}_{A}(V)$ is a rational polytope then we have
	\[\mathcal{E}(C)=\{\Delta \in C: K_{X}+\Delta \text{ is pseudoeffective}\}\]
	\[\mathcal{N}(C)=\{\Delta \in C: K_{X}+\Delta \text{ is nef}\}\]
	
	Given a birational contraction $\phi:X \dashrightarrow Y$ we also define
	\[\mathcal{W}_{\phi}(C)=\{\Delta \in \mathcal{E}(C): \phi \text{ is a weak log canonical (wlc) model of } (X,\Delta)\}\]
	and given a rational map $\psi:X \dashrightarrow Z$
	\[\mathcal{A}_{\phi}(C)=\{\Delta \in \mathcal{E}(C): \phi \text{ is the ample model of } (X,\Delta)\}\]
\end{definition}

	\begin{remark}
	
	As defined above, $\mathcal{RL}_{A}(V)$ is non-empty only when $(X,A)$ is log canonical. We might wish to allow $(X,A)$ to be rlc with fixed witnesses instead. This quickly becomes non-trivial because of the overlap of sets in the corresponding open cover.
	
	If we're interested in a pair $(X,A+B)$ where $(X,B)$ is rlt and $A$ is big and nef then for suitably small $t>0$, and some coefficient space $V$, we always have that $(X,tA+(1-t)A+B)$ is rlt with coefficients in $\mathcal{RL}_{tA}(V)$ by \autoref{bertini}. Moreover if we have finitely many such pairs, we can find $t,V$ suitable for all of them. This is normally enough in practice.
	
\end{remark}

We consider $X \to T$ to be part of the definition of $X$ and omit any mention of $T$ from the notation for rlt polytopes.
	
\begin{lemma}
	Take $A \geq 0$ and let $V$ be a coefficient space. Let $C \subseteq \mathcal{RL}_{A}(V)$ be a rational polytope. Then there is an open cover $U_{i}$ such that every $\Delta \in C$ is witnessed over $U_{i}$. If $C$ is an rlt polytope then we may choose $U_{i}$ such that every witness is klt.
\end{lemma}
\begin{proof}
	We can take the vertices $D_{i}$ of $C$. Then take witnesses $(X_{i,j}, B_{i,j})$ of $D_{i}$. Since there are finitely many $D_{i}$, we can assume that for all $i$ we have $X_{i,j}=X_{j}$ for some $X_{j}$ not depending on $i$, after taking intersections of combinations of the $X_{i,j}$ and renumbering as necessary. Now $C$ is the convex hull of the $D_{i}$ and $\Delta= \sum \lambda_{i}D_{i}$ has witnesses $\Delta_{j}= \sum \lambda_{i}B_{i,j}$ as required.
\end{proof}

Note that if $C$ is not an rlt polytope and $\Delta \in C$ is an rlt boundary, it might be that the above lemma gives only log canonical witnesses on each $U_{i}$. 

We will essentially only ever work with rational polytopes containing a klt boundary. Since the questions are always local we can normally assume these polytopes are simplices. By the following lemma, it is then enough to work with rlt polytopes.

\begin{lemma}\label{rlt-repl}
	Suppose $A$ is ample, $V$ is a coefficient space and that $C\subseteq \mathcal{RL}_{A}(V)$ is a rational simplex. If there is some boundary $B_{0} \in \mathcal{RL}_{A}(V)$ with $(X,B_{0})$ rlt, then there is an affine bijection $f:C \to C'$, where $C'$ is an rlt polytope inside $\mathcal{RL}_{A/2}(W)$ for some coefficient space $W$. Further $f, f^{-1}$ preserve rationality and $\mathbb{Q}$-linear equivalence.
\end{lemma}

\begin{proof}
	To show a rational polytope $C'\subseteq \mathcal{RL}_{A'}(V')$ is rlt it is enough to show that every vertex boundary $B_{i}$ of $C'$ is rlt with witnesses in $V'$.
	
	Indeed if this is the case then for $B \in C'$ we have $B= \sum \lambda_{i} B_{i}$ for $\lambda_{i} \geq 0$ with $\sum \lambda_{i}=1$. Let $U_{j}$ be an open cover such that each $B_{i}$ is witnessed by $(X_{j},B_{i,j})$, then $B|_{X_{j}}\sim \sum \lambda_{i}B_{i,j}$, so $(X,B)$ must be rlt as claimed.

	Write the vertices of $C$ as $B_{i}=A+\Delta_{i}$ for $i > 0$ and let be $B_{0}=A+\Delta_{0} \in \mathcal{RL}_{A}(V)$ be the rlt boundary. Now choose $\Gamma_{i} =(1-t_{i})\Delta_{i}+t_{i}\Delta_{0}$ for $t_{i}$ rational and sufficiently small that $\frac{A}{2}+t_{i}(\Delta_{i}-\Delta_{0})$ is ample. By construction $(X,A+\Gamma_{i})$ is rlt.
	
	Further choose $H_{i} \sim_{\mathbb{Q}} \frac{A}{2}+t_{i}(\Delta_{i}-\Delta_{0})$ effective and sharing no support with $A$. Then by construction
	\[A+\Delta_{i} \sim_{\mathbb{Q}} \frac{A}{2}+\Gamma_{i}+H_{i}=D_{i}\]
	and $(X,D_{i})$ is rlt by \autoref{bertini}. Reselecting $H_{i}$ if needed we may suppose that $D_{i}$ is not in the span of $\{D_{j}: i \neq j\}$ for each $i$. This can always be done since the $H_{i}$ are all ample.
	
	Let $W$ be a coefficient space containing the components of $\Delta_{i}, H_{i}$ such that each $(X,D_{i})$ is rlt with witnesses in $W$. Now let $C'$ be the convex hull of the $D_{i}$, so that $C'$ is an rlt polytope inside $\mathcal{RL}_{A}(W)$.
	
	Since $C$ is a simplex, by assumption, we can write any $B \in C$ uniquely as $B=\sum \lambda_{i} B_{i}$ where $\lambda_{i} \geq 0$ and $\sum \lambda_{i} =1$. Therefore, we can define a bijective affine map $f: C \to C'$ by sending $B_{i}=A+\Delta_{i} \to D_{i}$ and then writing $f(B)= \sum \lambda_{i} D_{i}$.
	
	Clearly $B$ is rational if and only if $\lambda_{i} \in \mathbb{Q}$, which happens if and only if $f(B)=\sum \lambda_{i} D_{i}$ is rational. So $f, f^{-1}$ preserve rationality. Equally as $B_{i} \sim_{\mathbb{Q}} D_{i}$ we must have $B \sim_{\mathbb{Q}} f(B)$, and the same holds for $f^{-1}$.
	
\end{proof}

\begin{remark}
	With the notation of \autoref{rlt-repl}, if $S \subseteq C$ is a rational polytope then $f(S)$ is also a rational polytope since $f$ is affine and preserves rationality. The converse is also true since $f^{-1}$ is also still affine and $f^{-1}f(S)=S$ as $f$ is a bijection.  
\end{remark}

Given a general rlc polytope we can always take a rational triangulation and define a piecewise affine bijection, $f$, by using the above procedure on each simplex. However, this does not in general preserve convexity, so it easier in practice to work locally on the polytope and assume it is a simplex. Alternatively, this could be remedied by working with $C'$, the convex hull of $f(C)$, since this must still be an rlt polytope. Then $f\colon C \to C'$ is no longer a bijection, but it is still preserves rationality and $\mathbb{Q}$-linear equivalence so would suffice for applications.

\begin{definition}
	Take $S, S' \subseteq \mathcal{RL}_{A}(V)$. We say $S \sim_{\mathbb{R}} S'$ if for every $\Delta \in S$ there is $\Delta' \in S'$ with $\Delta \sim_{\mathbb{R}} \Delta'$ and vice versa. The linear closure of $S$ is given by $$S^{*}=\bigcup_{S' \sim S}S'= \{\Delta \in \mathcal{RL}_{A}(V) \text{ such that } \exists \Delta' \in S \text{ with }\Delta \sim_{\mathbb{R}} \Delta'\}$$.
\end{definition}

\begin{lemma}
	Let $V$ be a finite dimensional, rational affine subspace of $WDiv_{\mathbb{R}}(X)$ and fix $A \geq 0$. Take $S \subseteq \mathcal{RL}_{A}(V)$ a rational polytope. Then the linear closure, $S^{*}$ is also a rational polytope. 
\end{lemma}

\begin{proof}
	By translating by $-A$ we can view $S$ as a subset of $V$. Similarly, after a translation by say $D$ of $V$ we can suppose that $V$ is a vector space. After these transformations we have that $S^{*}=\{B+E \text{ such that } B\in S, E \sim 0 \text{ and } B+E -D \geq 0\}$.

	Let $N=\{E \in V: E \sim_{\mathbb{R}} 0\}$ and take $\phi:V \to W=V/N \subseteq \text{Pic}(X)\otimes \mathbb{R}$, then $\phi(S)=\phi(S^{*})$ is a rational polytope in $W$ and its preimage $S+N$ is still cut out by finitely rational half spaces, but is no longer compact. Hence we must have that $S^{*}=(S+N)\cap({\Delta \geq D})$ is cut out by finitely many rational half spaces. 
	
	However for each point $B \in S$, the set $\{B\}^{*}=\{B+E\geq D \text{ such that } E \sim_{\mathbb{R}} 0\}$ is bounded, since the $E\in N$ such that $B+E \geq D$ are bounded by the coefficients of $B$ and $D$. Since $S$ is closed and bounded however we must have that $S^{*}$ is bounded too.
\end{proof}

In particular $\mathcal{RL}_{A}(V)$ is a rational polytope over a local ring, since it is the linear closure of $\mathcal{L}_{A}(V)$. To lift from the local case, we essentially find an open cover of $T$ which witnesses $\mathcal{RL}_{A}(V)$.

\begin{theorem}\label{rlt-poly}
	Let $V$ be a finite dimensional, rational affine subspace of $WDiv_{\mathbb{R}}(X)$ and fix $A \geq 0$. Then $\mathcal{RL}_{A}(V)$ is a rational polytope.
\end{theorem}

\begin{proof}
	For $W$ an affine subspace, let $\hat{W}=\{w-w' \text{ such that } w,w' \in W\}$.
	
	Take a point $p \in T$, and consider $X_{p}=X\times T_{p} \to T_{p}$. Let $A_{p},V_{p}$ be the restrictions of $A,V$ to $X_{p}$ and let $D_{i}$ be the vertices of $\mathcal{L}_{A_{p}}(V_{p})$, then there are open sets $U_{i}$ around $p$ such that $(X\times U_{i},D_{i})$ are lc when $D_{i}$ is extended over $U_{i}$. Moreover we may freely assume that there are no vertical components of $V$ which meet $U_{p}= \bigcap U_{i}$ but are not supported over $p$, thus ensuring for $E$ in $\hat{V}|_{X_{U_{p}}}$ where $X_{U_{p}}=X\times U_{p}$, we have $E \sim_{\mathbb{R}} 0$ if and only if $E|_{X_{p}}\sim_{\mathbb{R}} 0$. By compactness of $T$ there are finitely many $p_{j}$ such that $U_{j}=U_{p_{j}}$ is an open cover of $T$. 
	
	A pair $(X,\Delta)$ is rlc with witnesses in $V$ if and only if it is witnessed over $U_{j}$. Indeed if it is rlc, then we must be able to find $B_{j}$ such $(X_{p_{j}},B_{j})$ is lc and $B_{j} \sim_{\mathbb{R}} \Delta$. By construction however $B_{j}$ extends to an lc pair $(X_{j}=X\times U_{p_{j}},B_{j})$. Then $(X,\Delta)$ is witnessed by $(X_{j}, B_{j})$ as required.
	
	Consider $\mathcal{RL}_{A}(V)$, by the previous paragraph we may take an open cover $U_{i}$ such that every pair $(X,B) \in \mathcal{RL}_{A}(V)$ is witnessed by pairs $(X_{i}=X\times U_{i},B_{i})$. Let $C_{i} = \mathcal{L}_{A_{i}}(V_{i})^{*}$ where $A_{i}, V_{i}$ are the restrictions of $A,V$ to $X_{i}$ and write $S_{i}=\{\Delta \in V: \Delta|_{X_{i}} \in C_{i}\}$, then $\mathcal{RL}_{A}(V)= \bigcap S_{i}$ is a rational polytope since each $C_{i}$ is and there are no divisors $D \neq 0$ with $D|_{X_{i}} \neq 0$ for every $i$.
\end{proof}

In particular then $\mathcal{RL}_{A}(V)$ is closed. Moreover since it is a polytope, if $(X,\Delta_{i})$ is a sequence of rlc pairs with $\Delta_{i} \to \Delta$, then the witnesses of $\Delta$ may be chosen to be the limit of witnesses of $\Delta_{i}$

	\section{Finiteness of Log Terminal Models}
	

	In this section we prove our Finiteness of Minimal Models result. Here, as in \autoref{setup}, $R$ will always be an excellent ring with dualising complex, $T$ will be a quasi-projective $R$ scheme and $X$ will always be an integral scheme projective over $T$. All pairs will be considered as $R$ pairs over $T$.
	
	\begin{lemma}\label{neftope}
		Fix a $\mathbb{Q}$-divisor $A \geq 0$ and let $C\subseteq \mathcal{L}_{A}(V)$ be a rational polytope. Then $\mathcal{N}(C)=\{\Delta \in C \text{ such that } K_{X}+\Delta \text{ is nef } \}$ is also a rational polytope.
	\end{lemma}

	\begin{proof}	
		Let $B_{i}$ be the vertices of $C$. If $B \in C$ then $B= \sum \lambda_{i} B_{i}$ for $1 \geq \lambda_{i} \geq 0$ so $(K_{X}+B).C <0$ ensures $(K_{X}+B_{i}).C <0$ for some $i$. In particular if $R_{i,j}$ are the $K_{X}+B_{i}$ negative extremal rays then $K_{X}+B$ is nef if and only if $(K_{X}+B).R_{i,j} \geq 0$ for all $i,j$. Indeed, suppose that we have such a $K_{X}+B$ and that $R$ is a $K_{X}+B$ negative extremal ray, then $(K_{X}+B_{i}).R <0$ for some $i$ and so $R=R_{i,j}$ for some $j$, a contradiction. Then the condition $(K_{X}+B).R_{i,j} \geq 0$ defines a rational polytope by \cite[Proposition 9.31]{bhatt2020globally+}.
	\end{proof}
	
	Since this result does not require $A$ to be ample, we may often avoid the use of Bertini's Theorem, \cite[Lemma 3.7.3]{birkar2010existence} in particular, to substitute a big divisor for an ample one. Versions of these results are available for rlt polytopes but making use of them requires extra back and forth between the klt and rlt case. 
	
	\begin{lemma}\label{rationality}
		Let $\phi: X \dashrightarrow Y$ be a birational contraction. Let $C \subseteq \mathcal{RL}_{A}(V)$ be an rlt polytope, then $\mathcal{W}_{\phi}(C)$ is a rational polytope.
	\end{lemma}
	\begin{proof}
		We can choose a finite open cover, $U_{i}$ such that $C$ is witnessed by klt pairs over $U_{i}$. On $X_{i}$ we can write $N_{i}=\{E\sim_{\mathbb{R}} 0\} \subseteq V_{i}=V|_{X_{i}}$, $C_{i}=C|_{X_{i}}$ and consider the induced map $\phi_{i}:X_{i} \to Y_{i}$. Now let $C'_{i}=\mathcal{L}_{A_{i}}(V_{i}) \cap C_{i}^{*}$. After perhaps shrinking $C'_{i}$ we may suppose it is a klt polytope and $C_{i} \subseteq (C')^{*}_{i}$. Thus $\mathcal{W}_{\phi_{i}}(C'_{i})$ is a rational polytope by \cite[Corollary 3.11.2]{birkar2010existence} with \cite[Theorem 3.11.1]{birkar2010existence} and \cite[Lemma 3.7.4]{birkar2010existence} replaced by Lemma \ref{neftope}.
		
		Therefore $\mathcal{W}_{i}=\mathcal{W}_{\phi_{i}}(C_{i})=\mathcal{W}_{\phi_{i}}(C'_{i})^{*}\cap C_{i}$ is also a rational polytope. For each $\mathcal{W}_{i}$ we have a rational polytope $\hat{\mathcal{W}_{i}}=\{\Delta \in C: \Delta|_{X_{i}} \in \mathcal{W}_{i}\} \subseteq C$. The intersection of these polytopes is precisely $\mathcal{W}_{\phi}(C)$.
	\end{proof}
	
	\begin{lemma}\label{faces}
		Let $\phi: X \dashrightarrow Y$ be a birational contraction. Let $C$ be an rlt polytope, let $F \subseteq \mathcal{W}_{\phi}(C)$ be a face, possibly with $F =\mathcal{W}_{\phi}(C)$. Suppose $f: X \dashrightarrow Z$ is an ample model for some $B$ in the interior of $F$. Then there is a factorisation $f=g \circ \phi$ for some morphism $g:Y \to Z$, and moreover $f$ is an ample model for every boundary in the interior of $F$.
	\end{lemma}
	
	\begin{proof}
		
		Since $\phi$ is a wlc model for $B$ we have an induced map $g: Y \to Z'$. However then $g\circ \phi$ is an ample model for $(X,B)$, so after post-composition with an isomorphism we may suppose $Z=Z'$ and $f=g\circ \phi$. Suppose $B' \in \mathcal{W}_{\phi}(C)$ then $f$ is an ample model for $(X,B')$ if and only $g$ is an ample model for $(Y,\phi_{*}B')$. Since $K_{Y}+\phi_{*}B$' is semiample $g$ is an ample model if and only if the curves contracted by $g$ are precisely those $\Gamma$ with $(K_{Y}+\phi_{*}B').\Gamma=0$.
		
		Suppose then $B'$ is in the interior of $F$. Consider $B_{t}=tB+(1-t)B'$, so that $K_{Y}+\phi_{*}B_{t}= t(K_{Y}+\phi_{*}B)+(1-t)(K_{Y}+\phi_{*}B')$. Then if $(K_{Y}+\phi_{*}B').\Gamma \neq 0$ then and $(K_{Y}+\phi_{*}B).\Gamma=0$ it must be that $(K_{Y}+\phi_{*}B_{t}).\Gamma < 0$ for all $t < 0$. However for small $t$ we have $B_{t} \in F$, a contradiction. By symmetry, we see that $\Gamma$ is contracted by $g$ if and only if $(K_{Y}+\phi_{*}B').\Gamma=0$, so $f$ is an ample model for $K_{X}+\phi_{*}B'$ also.
	\end{proof}

\begin{theorem}\cite[Theorem 9.33]{bhatt2020globally+}
	Suppose that $X$ is $\mathbb{Q}$-factorial and let $C$ be a klt polytope in $\mathcal{L}_{A}(V)$ for $A\geq 0$ big. There is a finite collection of log terminal models $\phi_{i}: X \dashrightarrow Y_{i}$ such that every $B \in \mathcal{E}(C)$ has some $j$ with $\phi_{j}$ a log terminal model of $(X,B)$. 
\end{theorem}
	
	\begin{corollary}\label{klt_finiteness}
	Suppose that $X$ is $\mathbb{Q}$-factorial and let $C$ be a klt polytope with $A$ big. Suppose that every $B \in C$ has components which span $NS(X)$, then there are finitely many birational maps $\phi_{i}: X \dashrightarrow Y_{i}$ such that for any $B \in \mathcal{E}(C)$ if $\phi: X \dashrightarrow Y$ is a wlc model then $\phi_{i}=f \circ \phi$ for some $i$ and some isomorphism $f: Y \to Y_{i}$.
\end{corollary}

\begin{proof}
	
	After possibly expanding $V$, we can take $C'\subseteq \mathcal{L}_{\frac{A}{2}}(V)$ a klt polytope with $C \subseteq C'$ such that for any $B \in C$ if $D$ is a component of $B$ then $B+tD$ is in $C'$ for any $|t| < \epsilon$, for some $\epsilon >0$ depending only on $B$. This can be done by taking $C'$ to be the convex hull of small perturbations of the vertices of $C$.
	
	By the previous theorem there are finitely many birational maps $\phi_{i}: X \dashrightarrow Y_{i}$ such that for every $B \in \mathcal{E}(C')$ there is some $\phi_{i}$ a log terminal model of $(X,\Delta)$.
	
	Further are then finitely many morphisms $f_{i,j}:Y_{i}\to Z_{j}$ such that $\psi_{i,j}=f_{i,j} \circ \phi_{i}$ are ample models such that $B \in \mathcal{E}(C')$ some $\psi_{i,j}$ is the (unique) ample model of $(X,B)$. This is because the $f_{i,j}$ correspond to faces of the rational polytope $\mathcal{W}_{\phi_{i}}(C')$ by \autoref{faces}.
	
	Now pick $\Delta \in C$. Let $\psi: X \dashrightarrow Y$ be a wlc model for $\Delta$. We can take $D$ in the span of the components of $B$ such that $\phi$ is $B+D$ negative and $\phi_{*}D$ is ample. By shrinking $D$, we can suppose that $B+D\in C'$. Thus we have that $\psi$ is the ample model of some $B+D \in \mathcal{W}_{\psi}(C')$. Now take a log terminal model of $B+D$ of the form $\phi_{i}$ for some $i$,. By uniqueness of the ample model, up to post-composition with an isomorphism, we have $\psi=f_{i,j} \circ \phi_{i}=\psi_{i,j}$ for some $j$. 
	Thus the family of models $\{\psi_{i,j}\}$ give the required maps.
	
\end{proof}

\begin{theorem}\label{weak finiteness}
	Let $A$ be a big $\mathbb{Q}$-divisor and chose $V$ a coefficient space. Take $C$ be an rlt polytope inside $\mathcal{RL}_{A}(V)$, then
	
	\begin{enumerate}
		\item There are finitely many birational maps $\phi_{j}: X \dashrightarrow Y_{j}$ such that for any $B \in \mathcal{E}(C)$ if $\phi: X \dashrightarrow Y$ is a wlc model then $\phi_{j}=f \circ \phi$ for some $j$ and some isomorphism $f: Y \to Y_{j}$. \\
		\item There are finitely many rational maps $\psi_{k}: X \dashrightarrow Z_{k}$ such that if $\psi:X \dashrightarrow Z$ is an ample model for some $B \in \mathcal{E}(C)$ then there is an isomorphism $f:Z \to Z_{k}$ with $\psi_{k}=f \circ \psi_{k}$.
	\end{enumerate}
\end{theorem}

	\begin{proof}
	
	We prove 1., 2. follows immediately as ample models correspond to the interiors of faces of the $\mathcal{W}_{\phi_{i}}(C)$ by \autoref{faces}.\\	
	
	Equally, it is enough to show this in the case that $C$ is a klt polytope. Indeed suppose it holds for klt polytopes. Then take an open cover $U_{i}$ of $T$ witnessing $C$. For each $i$ we may take a klt polytope $C'_{i}$ with $C'_{i} \sim C_{i}=C|_{U_{i}}$. Given a wlc map $\phi:X \dashrightarrow Z$ for $B \in \mathcal{E}(C)$, we can let $\phi_{i}$ be the induced map on $X_{i}$ which is a wlc model for some $B_{i} \in C'_{i}$. In particular for fixed $i$ there are finitely many $\phi_{i,j}$ such that for any $B$ and $\phi$ we have $f_{i} \circ \phi_{i}=\phi_{i,j}$ for some $j$ and $f_{i}$. As $U_{i}$ is a finite cover there are finitely many $\phi_{i,j}$ indexed over $i,j$.
	
	If we have another map $\Phi:X \dashleftarrow Z'$ with isomorphisms $g_{i}$ such that $f_{i} \circ \phi_{i}=\phi_{i,j}= g_{i} \circ \Phi_{i}$, then $h_{i}=g_{i} \circ f^{-1}_{i}$ glues to an isomorphism $Z' \to Z$ over $T$. Thus there are only finitely many wlc models up to isomorphism.
	
	Suppose then that $C$ is a klt polytope.		
	
	Let $\pi:Y \to X$ be a log resolution of the support of $V$. Then for any $\Delta$ in $C$ we have $\pi^{*}(K_{X}+\Delta)+E=(K_{Y}+\Delta')$ where $E \geq $ is exceptional and shares no components with $\Delta'$ and $(Y,\Delta')$ is klt. Sending $\Delta \to \Delta'$ as above we can find a new polytope $C'$ on which it is sufficient to check the result holds. By replacing $C$ with $C'$, $A$ with $\pi^{*}A$, $X$ with $Y$ and $V$ with a suitable space, we may suppose that $X$ is regular, though it may no longer be the case that $A$ shares no support with $V$. 
	
	Let $H_{k}$ be ample divisors spanning $NS(X)$ and sharing no components with $A$ or $V$. Let $H= \sum H_{k}$. Note that for any open $U$ in $T$ we still have the components of $H|_{X_{U}}$ span $NS(X_{U})$, since $NS(X)$ surjects on $NS(X_{U})$ by $\mathbb{Q}$-factoriality of $X$.
	
	After shrinking $H$ we may take some $A',E \geq 0$ and a small $t>0$ such that:
	
	\begin{itemize}
		\item $E \geq 0$ shares no components with $H$;
		\item $A-E$ is ample;
		\item $t(A-E)-H \simeq A' > 0$ is ample and shares no components with $V,H$ or $E$;
		\item $\{A+H+B+tE \colon A+B \in C\}$ is a klt polytope; and
		\item $C'=\{A'+(1-t)A+H+B+tE \colon A+B \in C\}$ is an rlt polytope.
	\end{itemize}
	That we can choose $C'$ to be rlt follows from \autoref{bertini}. Note that $A'+(1-t)A+H+B+tE \simeq A+B$ by construction. Thus it suffices to check the result for $C'$ since $C'\subseteq \mathcal{L}_{H}(W)$ for some coefficient space $W$. As above, by taking an open cover, we may in fact assume that $C'$ is klt. But then the result follows by \autoref{klt_finiteness}, since the components of $H$ span $NS(X)$ by construction.	
	
\end{proof}

\begin{theorem}\label{rltfiniteness}
	Let $A$ be an ample $\mathbb{Q}$-Cartier divisor and $C$ be a rational polytope inside $\mathcal{RL}_{A}(V)$. Suppose there is a boundary $A+B \in \mathcal{RL}_{A}(V)$ such that $(X,A+B)$ is rlt with witnesses in $V_{A}$. Then the following hold:
	
	\begin{enumerate}
		\item There are finitely many birational contractions $\phi_{i}:X \dashrightarrow Y_{i}$ such that 
		\[\mathcal{E}(C) = \bigcup \mathcal{W}_{i}=\mathcal{W}_{\phi_{i}}(C)\]
		where each $\mathcal{W}_{i}$ is a rational polytope. Moreover if $\phi:X \to Y$ is a wlc model for any choice of $\Delta \in \mathcal{E}(C)$ then $\phi=\phi_{i}$ for some $i$, up to composition with an isomorphism.
		
		\item There are finitely many rational maps $\psi_{j}:X \dashrightarrow Z_{j}$ which partition $\mathcal{E}(C)$ into subsets $\mathcal{A}_{\psi_{j}}(C)=\mathcal{A}_{i}$.
		\item  For each $W_{i}$ there is a $j$ such that we can find a morphism $f_{i,j}: Y_{i} \to Z_{j}$ and $W_{i} \subseteq \overline{A_{j}}$.
		\item  $\mathcal{E}(C)$ is a rational polytope and $A_{j}$ is a union of the interiors of finitely many rational polytopes.
	\end{enumerate}
	
	If $C$ is an rlt polytope then $A$ big suffices.
\end{theorem}

\begin{proof}
	
	Since the convexity condition of every sub-polytope in the theorem statement is clear, it is enough to show that the result holds for every simplex in a rational triangulation of $C$. Thus after extending $V$ and changing $A$ as needed we may suppose:
	
	\begin{itemize}
		\item $C$ is a simplex;
		\item $C$ is an rlt polytope by \autoref{rlt-repl};
		\item $\mathcal{E}(C)$ is covered by $\mathcal{W}_{\phi_{i}}(C)$  and has a decomposition into disjoint sets $\mathcal{A}_{\psi_{j}}(C)$
		for some collection of birational contractions $\phi_{i}$ and rational maps $\psi_{j}$ by \autoref{rltmmp}; and
			\item There are only finitely many $\phi_{i}$ and $\psi_{j}$ by \autoref{weak finiteness}.
			\end{itemize}
			
			Take one of the wlc models $\phi_{i}:X \dashrightarrow Y_{i}$ , then just as in Lemma $\ref{faces}$, if $\Delta,\Delta'$ are in the same face of $\mathcal{W}_{i}$ then they have the same ample model. In particular then let $\psi_{j}:X \dashrightarrow Z_{j}$ be the ample model corresponding to the interior of $\mathcal{W}_{i}$, then we have a morphism $f_{i,j}: Y_{i} \to Z_{j}$ and  $W_{i} \subseteq \overline{A_{j}}$ as required. 
			
			Similarly by \autoref{faces} we have that $A_{j} \cap \mathcal{W}_{i}$ is a union of the interiors of some faces of $\mathcal{W}_{i}$. Since there are finitely many $\mathcal{W}_{i}$ and they cover $\mathcal{E}(C)$ the result follows.
		\end{proof}
		
		\begin{remark}
			In practice since we can always extend $V$ and $C$ it is enough to know that $(X,A)$ is klt, rather than needing an rlt pair $(X,A+B)$. Similarly if $X$ is klt, we can always find $t>0$ such that $(X,tA)$ is klt. Then if $(X,A+B)=(X,tA+(1-t)A+B)$ is rlc with coefficients in $V_{A}$ it is also rlc with witnesses in $V'_{tA}$ for some coefficient space $V'$. By choosing $V'$ such that all the vertices of $C$ are rlc with witness in $V'_{tA}$, we see that it is enough to suppose that $X$ is klt.
		\end{remark}
		
\section{Geography of Ample Models}

We keep the notation of the previous section, though we denote the closure of $\mathcal{A}_{\phi}(C)$ by $\mathcal{D}_{\phi}(C)$. As always $R$ will be an excellent ring with dualising complex, $T$ will be a quasi-projective $R$ scheme and all other schemes will be integral and projective, surjective morphism to $T$ over $R$. All pairs will be considered as $R$ pairs over $T$. Unlike in previous sections, we will work with $A$ ample throughout.

We will say the span of a polytope $C$ is $$\text{Span}(C)=\{\lambda(B-B') \text{ such that } B, B' \in C \text{ and } \lambda \in \mathbb{R}\}.$$  In a slight abuse of notation we say that $C \subseteq WDiv(X)$ spans $NS(X)$ if the span of $C$ surjects onto $NS(X)$. Equivalently this means if $D$ is a divisor and $B$ is in the interior of $C$ then for all sufficiently small $t>0$ $B+tD\equiv D'_{t}$ for some $D'_{t}\in C$.

\begin{lemma}\label{close}
	Let $X\to T$ be a $\mathbb{Q}$-factorial, klt threefold over $R$. Let $\phi:X \dashrightarrow Y$ be a wlc model of an rlc pair $(X,\Delta)/T$. Let $A \geq 0$ be an ample $\mathbb{Q}$-divisor and $C$ be a polytope inside $\mathcal{L}_{A}(V)$. Then we have that $\mathcal{D}_{\phi}(C):=\overline{\mathcal{A}_{\phi}(C)} \subseteq \mathcal{W}_{\phi}(C)$ is a rational polytope, moreover if $C$ spans $NS(X)$ and contains an open set around $\Delta$ then this inclusion is an equality.
\end{lemma}

\begin{proof}
	Suppose that $B \in \mathcal{A}_{\phi}(C)$. Then by \cite[Theorem 3.6.5]{birkar2010existence} we see that in fact $\phi$ is a wlc model for $B$ and thus we have $\mathcal{A}_{\phi}(C) \subseteq \mathcal{W}_{\phi}(C)$. So $\overline{\mathcal{A}_{\phi}(C)}$ is a union of faces of $\mathcal{W}_{\phi}(C)$ by \autoref{faces} and \autoref{rlt-repl}. However $\mathcal{A}_{\phi}$ is convex inside $\mathcal{W}_{\phi}(C)$ so it must be that $\overline{\mathcal{A}_{\phi}(C)}$ is a face of $\mathcal{W}_{\phi}(C)$, and thus is a polytope.
	
	Now suppose $C$ spans $NS(X)$ and contains an open set around $\Delta$.
	Let $H$ be a general ample divisor on $Y$. Let $W$ be a common resolution with maps $p:W \to X$, $q:W \to Y$. Then by assumption there is some $H' \equiv p_{*}q^{*}H$ with support contained in the support of $\Delta$, and hence in the support of any $B$ in the interior of $\mathcal{W}_{\phi}(C)$. Take such a $B$, then there is $\epsilon >0$ with $(X,B+\epsilon H') \in C$ , for any $\epsilon' \in ((0,\epsilon])$ $\phi$ is an ample model of $(X,B+ \epsilon' H')$, such an $\epsilon$ exists since $\phi$ is necessarily $H'$ negative. Thus $B+\epsilon' H' \in \mathcal{A}_{\phi}(C)$. But then we must have $\mathcal{W}_{\phi}(C)\subseteq \overline{\mathcal{A}_{\phi}(C)}$.
\end{proof}

\begin{theorem}\label{assumptions}\cite[Theorem 3.3]{hacon2009sarkisov}
	Let $C$ be a polytope inside $\mathcal{RL}_{A}(V)$, then there are finitely many maps $f_{i}:X \dashrightarrow Y_{i}$ over $T$ with the following. properties.
	
	\begin{enumerate}
		\item $\{A_{i}=A_{f_{i}}(C)\}$ partition $\mathcal{E}(C)$. If $f_{i}$ is birational then $\mathcal{D}_{i}=\mathcal{D}_{\phi_{i}}(C)$ is a rational polytope.
		\item If $A_{j} \cap \mathcal{D}_{i} \neq \emptyset$ then there is a morphism $f_{i,j}:Y_{i} \to Y_{j}$ such that $f_{i}=f_{i,j} \circ f_{j}$.
	\end{enumerate}
	
	Moreover if $C$ spans $NS(X)$ then we also have the following.
	
	\begin{enumerate}
		\item[3.] Pick $i$ such that a connected component, $\mathcal{D}$ of $\mathcal{D}_{i}$ meets the interior of $C$. Then the following are equivalent:
		\begin{itemize}
			\item $\dim \mathcal{D}= \dim C$.
			\item $\text{Span}(\mathcal{D})=\text{Span}(C)$.
			\item If $B \in \mathcal{A}_{i} \cap \mathcal{D}$ then $f_{i}$ is a log terminal model of $(X,B)$.
			\item $f_{i}$ is birational and $X_{i}$ is $\mathbb{Q}$-factorial.
		\end{itemize} 
		
		\item[4.] Suppose that $\mathcal{D}_{i}$ has the same span as $C$ and $B$ is a general point in $\mathcal{A}_{j} \cap \mathcal{D}_{i}$. If in fact $B$ is in the interior of $C$ then the relative picard number of $Y_{i}/Y_{j}$ is the difference in dimension of $\mathcal{D}_{i}$ and $\mathcal{D}_{j} \cap \mathcal{D}_{i}$. 
		
	\end{enumerate}
\end{theorem}

This result is stated for $C=\mathcal{L}_{A}(V)$ in characteristic zero, but the proof goes through essentially verbatim in this setting.

For brevity we fix some notation, essentially due to Shokurov.

\begin{definition}
	Take a coefficient space $V$, an ample divisor $A$ and then let $C$ be a polytope inside some $\mathcal{RL}_{A}(V)$.
	
	Suppose that $3$ and $4$ of the previous lemma hold for $C$, then the triple $(C,A,V)$ is a said to be a geography, when $A$ and $V$ are clear we sometimes just call $C$ a geography. The dimension of $(C,A,V)$ will be the dimension of $C$.
	The $\mathcal{D}_{\phi}$ are called classes.
	If $C$ is a geography and $\dim \mathcal{D_{\phi}}= \dim C$ then $\mathcal{D}_{\phi}$ is said to be a country. The codimension $1$ faces of countries are called borders, and a codimension $2$ face is called a ridge.
	If $(X,B)$ is a pair such that every country in $C$ is induced by a log terminal model of $(X,B)$ then $(C,A,V)$ is a geography for $(X,B)$.
\end{definition}

\autoref{assumptions} then says that if $(C,A,V)$ is a triple such that $C$ spans $NS(X)$ then $C$ is a geography.
This combined with following will be the main method of producing geographies for the remainder of the section.

\begin{lemma}
	Let $(C,A,V)$ be a geography. Take $W \subseteq V$ be a general coefficient space and let $W_{A}=\{A+B, B \in W\}$ then $C'=C \cap W_{A}$ is a geography. 
\end{lemma}
\begin{proof}
	Index all of the faces of every polytope in the decomposition by $\mathcal{D}_{i}$ as $F_{j}$. Then for $C'$ to be a geography it is enough to know that intersecting with $W$ preserves the codimension of the $F_{j}$ meeting $W$. For fixed $j$, however the choices of $W$ such that either $W$ does not meet $F_{j}$ or $F'_{j}=F_{j} \cap W_{A} \subseteq C'$ has the same codimension as $F \subseteq C_{A}$ form an open set in the Grassmanian. Since there are finitely many faces the result holds for suitably general choice of $W$.
\end{proof}

\begin{lemma}
	Suppose $V$ is a coefficient space which spans $NS(X)$. Let $C$ be any polytope contained $\mathcal{RL}_{A}(V)$, then after perturbing the vertices by an arbitrarily small amount $(C,A,V)$ is a geography.
\end{lemma}
\begin{proof}
	Since we can perturb the vertices of $C$ we may suppose it is rational and contained in the interior of $\mathcal{RL}_{A}(V)$. Let $W$ be the minimal coefficient space in $V$ with $C \subseteq W_{A}\cap \mathcal{RL}_{A}(V)$. Since $C$ is contained in the interior of $\mathcal{RL}_{A}(V)$, we can pick an rlt polytope $C'$ which spans $NS(X)$ with $W_{A}\cap C'=C$. Then after a small perturbation of the vertices we may suppose that $W_{A}\cap C'$ is a geography, as required.
\end{proof}

\begin{lemma}\cite[Lemma 3.6]{hacon2009sarkisov}\label{amp}
		Let $(X,\Delta)/T$ be an rlt threefold pair and $f\colon X \dashrightarrow Y$ a birational contraction of $\mathbb{Q}$-factorial projective $T$-schemes. Suppose that $B-\Delta$ is ample and $f$ is an ample model for $K_{X}+B$. Then $f$ is a log terminal model for $(X,\Delta)$.
\end{lemma}

\begin{lemma}\label{geo}
	Suppose that $f_{i}: (X,\Delta) \to (Y_{i},\Delta_{i})$ for $i=1,..n$ are a finite collection of $\mathbb{Q}$-factorial Mori Fibre spaces obtained by running an MMP for a rlt threefold pair $(X,\Delta)$ with $X$ regular. Then there is a geography $(C,A,V)$ for $(X,\Delta)$ of dimension at most $n$ such that every $\mathcal{D}_{f_{i}}$ is a country. 
	
	Moreover if $g_{i}:Y_{i}\to Z_{i}$ are the Mori Fibrations and we write $h_{i}=g_{i}\circ f_{i}$. Then we may choose $C$ such that $\mathcal{D}_{h_{i}}$ are borders of the $\mathcal{D}_{f_{i}}$ and their interiors are connected by a path through the border of $\mathcal{E}(C)$ contained entirely in the interior of $C$.
\end{lemma}

\begin{proof}
	
	Pick $A'_{i}$ ample on $Z_{i}$ such that $g_{i}^{*}A'_{i}-(K_{Y_{i}}+\Delta_{i})$ is ample. 
	
	We may choose $H$ ample on $X$ whose components span $NS(X)$ together with $A$ ample both sufficiently small such that:
	\begin{itemize}
		\item $(X,H+A)$ is klt,
		\item the $A_{i} = g_{i}^{*}A'_{i}-(K_{Y_{i}}+\Delta_{i}+f_{i,*}(A+H))$ are ample,
		\item $(X,\Delta+A+H)$ is an rlt pair which is not pseudo-effective, and
		\item each $f_{i}$ is $(K_{X}+\Delta+A+H)$ negative.
	\end{itemize}
	
	Further, we may pick $A$ such that it avoids the exceptional loci of the $f_{i}$ and shares no components with $H$.
	
	By \autoref{bertini} we can take $B_{i} \sim f_{i}^{*}A_{i}$ such that each $(X,\Delta+H+A+B_{i})$ is rlt. Moreover we can choose the $B_{i}$ such that they share no components with $A$ since the augmented base locus of $B_{i}$ is precisely the exceptional locus of $f_{i}$. Thus the $(X,\Delta+B_{i})$ all have witnesses in some $W$ for which $(X,\Delta+H+A+B_{i})$ have witnesses in $W_{A+H}$.
	
	By construction, then, after adding the components of $H$ to $W$ we have $(X,\Delta+B_{i}+H+A) \in \mathcal{RL}_{A}(W)$, a geography. Further the $f_{i}$ are wlc models of the $(X,\Delta+B_{i}+H+A)$ and the $h_{i}$ are the ample models.
	
	Let $C$ be the convex hull of the $\Delta+B_{i}+H+A$ and $\Delta+H$. Since the components of $H$ span NS(X), and the $f_{i}$ are wlc models for boundaries in $C$, we can find boundaries in $\mathcal{RL}_{A}(W)$ for which the $f_{i}$ is an ample model. Moreover we can find them arbitrarily close to $C$. Thus we can freely move the vertices of $C$ an arbitrarily small amount such that it meets the interior of each of the $\mathcal{D}_{f_{i}}$ and their borders $\mathcal{D}_{h_{i}}$ while ensuring they are sufficiently general that $C$ is a geography. 
	
	By construction, $C_{-\Delta}$ is contained in the ample cone and $\dim C \leq n$. It remains to check that $\mathcal{D}_{h_{i}}$ are borders of the $\mathcal{D}_{f_{i}}$ and their interiors are connected by a path through the border of $\mathcal{E}(C)$ contained entirely in the interior of $C$.
	
	Since $C$ contains a vertex $D=\Delta+H \notin \mathcal{E}(C)$ such that $C_{-D}$ is contained in the effective cone, it is enough to check that for each $i$ the interior of $\mathcal{D}_{h_{i}}$ meets the interior of $C$, but this again is ensured by the construction. Thus we may take $E_{i},E_{j}$ in the interiors of $\mathcal{D}_{h_{i}},\mathcal{D}_{h_{j}}$ respectively and both contained in the interior of $C$. Then the simplex formed by $D,E_{i},E_{j}$ meets the boundary of $\mathcal{E}(C)$ along a path connecting $E_{i}$ and $E_{j}$, wholly contained in the interior of $C$.
	
\end{proof}

\begin{lemma}\cite[Lemma 3.5]{hacon2009sarkisov}\label{links}
	Let $(C,A,V)$ be a geography on $X$ of dimension $2$. Take two ample classes $\mathcal{D}_{f}$ and $\mathcal{D}_{g}$ corresponding to some maps $f:X \dashrightarrow Y$ and $g: X\dashrightarrow Z$. Suppose that $\mathcal{D}_{f}$ is a country and that they meet along a border $\mathcal{B}$ not contained in the boundary of $C$. Suppose further that $\rho(Y) \geq \rho(Z)$
	
	Let $h: Y \dashrightarrow Z$ be the map induced by $\mathcal{B}$. Take $B$ an interior point of $\mathcal{B}$ and let $\Delta=f_{*}B$, then one of the following holds.
	\begin{enumerate}
		\item $\rho(Y)=\rho(Z)+1$ and $h$ is a $K_{Y}+\Delta$ trivial morphism. Thus either
		\begin{enumerate}[a)]
			\item $h$ is a divisorial contraction and $\mathcal{B} \neq \mathcal{D}_{g}$
			\item $h$ is a small contraction and $\mathcal{B}=\mathcal{D}_{g}$
			\item $h$ is a MFS and $\mathcal{B}=\mathcal{D}_{g}$ is contained in the boundary of $\mathcal{E}(C)$.
		\end{enumerate}
		\item $\rho(W)=\rho(Y)$ and $h$ is a $K_{Y}+\Delta$ flop and $\mathcal{B} \neq \mathcal{D}_{g}$ is not contained in the boundary of $\mathcal{E}(C)$.
	\end{enumerate}
\end{lemma}

\section{Sarkisov Program} \label{Sarkisov-section}

Fix a positive dimensional quasi-projective $R$ scheme, $T$. Suppose that $f:X \to Z$, $g:Y \to W$ are two Mori Fibre Spaces, projective and surjective over $T$. We say that they are Sarkisov related if they are both outputs of an MMP from the same $\mathbb{Q}$-factorial rlt pair. In particular we require $X,Y$ to be $\mathbb{Q}$-factorial. 

A Sarkisov link $s:X \dashrightarrow Y$ is one the following.

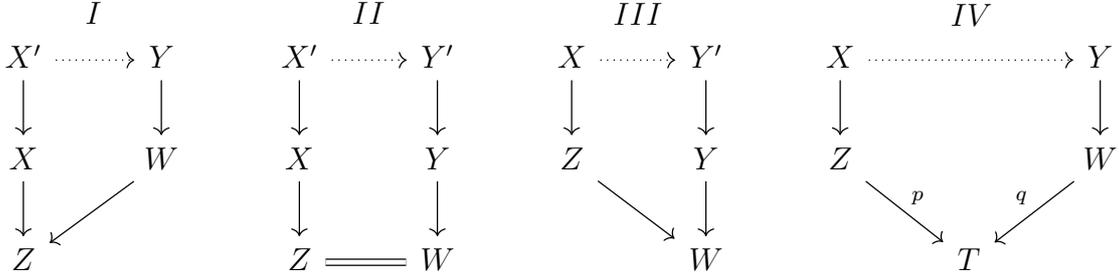
\[\begin{tikzcd}
X' \arrow[d] \arrow[r, dotted] \arrow[r, "I", phantom, bend left=49] & Y \arrow[d]  & X' \arrow[r, dotted] \arrow[d] \arrow[r, "II",phantom, bend left=49] & Y' \arrow[d] & X \arrow[r, dotted] \arrow[d] \arrow[r, "III", phantom, bend left=49] & Y' \arrow[d] & X \arrow[d] \arrow[rr, dotted] \arrow[rr, "IV",phantom, bend left] &   & Y \arrow[d]       \\
X \arrow[d]                                                          & W \arrow[ld] & X \arrow[d]                                               & Y \arrow[d]  & Z \arrow[rd]                                              & Y \arrow[d]  & Z \arrow[rd, "p"]                                          &   & W \arrow[ld, "q"'] \\
Z                                                                    &              & Z \arrow[r, equal]                                                         & W            &                                                           & W            &                                                            & T &                  
\end{tikzcd} \]

Such that the following holds:
\begin{itemize}
	\item There is an rlt pair $(X,\Delta)/T$ or $(X',\Delta')/T$ as appropriate such that the horizontal map is a sequence of flops for this pair
	\item Every vertical morphism is a contraction
	\item If the target of a vertical morphism is $X$ or $Y$ then it is an extremal divisorial contraction
	\item Either $p,q$ are both Mori Fibre Spaces (this is type $IV_{m}$) or they are both small contractions (type $IV_{s}$)
\end{itemize}

We realise these Sarkisov links inside two dimensional geographies as follows.

Fix $X\to T$ a threefold over $R$ and a geography $(C,A,V)$ on $X$ of dimension $2$.

Let $\Delta$ be a point in the boundary of $\mathcal{E}(C)$ but in the interior of $C$. Let $\mathcal{T}_{1}=\mathcal{D}_{f_{1}},..., \mathcal{T}_{k}=\mathcal{D}_{f_{k}}$ be the countries which meet $\Delta$. Let $\mathcal{B}_{i}$ be the borders $\mathcal{T}_{i}$ meeting $\Delta$ such that after reordering we have $\mathcal{B}_{i}=\mathcal{T}_{i}\cap \mathcal{T}_{i+1}$ for $1 \leq i \leq k-1$. Then $\mathcal{B}_{0}, \mathcal{B}_{k}$ are contained in the boundary of $\mathcal{E}(C)$. Let $g_{i}:X \to Z_{i}$ be the ample models associated to the interiors of $\mathcal{B}_{i}$

Relabel $\phi=f_{0}:X \dashrightarrow Y$, $Z=Z_{0}$, $\psi=f_{k}\dashrightarrow W$ and $T=Z_{k}$. Then we have $p,q$ with $p \circ \phi=g_{0}$ and $q \circ \psi =g_{k}$.


%
%

\begin{theorem}\cite[Theorem 3.7]{hacon2009sarkisov}\label{islinked}
	With notation as above, suppose $B$ is any divisor on $X$ with $\Delta-B$ ample. Then $q:Y \to Z$ and $q: W \to T$ are two Mori Fibre spaces obtained by running $(X,B)$ MMPs and they are connected by Sarkisov links.
\end{theorem}

\begin{theorem}\label{sarkisov}
		Fix an integral quasi-projective scheme $T$ over $R$. Let $g_{1}:Y_{1} \to Z_{1}$ and $g_{2}:Y_{2} \to Z_{2}$ be two Sarkisov related, klt Mori fibre spaces of dimension $3$, projective $T$. If the $Y_{i}$ have positive dimension image in $T$, then they are connected by Sarkisov links.
\end{theorem}

\begin{proof}
	
	By assumption these Mori fibre spaces are outputs of an MMP for some pair klt $(X,\Delta)/T$. Replacing $X$ with a suitable resolution, we may suppose that $X$ is smooth and admits morphisms $f_{i}:X \to Y_{i}$. Let $h_{i}=g_{i} \circ f_{i}$ then by Lemma $\ref{geo}$ there is a geography for $(X,\Delta)$ of dimension $2$ such that the $\mathcal{D}_{f_{i}}(C)$ are countries and the interiors of the $\mathcal{D}_{h_{i}}$ are connected by a path along the boundary of $\mathcal{E}(C)$. 
	
	Each ridge in this path corresponds to a Sarkisov link by \autoref{islinked}. Thus following the path gives a (non-unique) decomposition of $f_{2} \circ f_{1}^{-1} \colon Y_{1} \dashrightarrow Y_{2}$ into Sarkisov links. Since $\mathcal{E}(C)$ is a rational polytope, there are finitely many links.

\end{proof}
	
	\bibliography{Refs}
	\bibliographystyle{alpha}
\end{document}